\documentclass[reqno, 11pt]{tdb}

\usepackage{amsfonts, amsthm, amsmath, amssymb}

\usepackage{a4wide}
\usepackage{url}

\newtheorem{theorem}{Theorem}
\newtheorem{lemma}{Lemma}

\newtheorem{conjecture}{Conjecture}

\theoremstyle{definition}
\newtheorem*{ack}{Acknowledgements}

\renewcommand{\d}{\mathrm{d}}
\renewcommand{\phi}{\varphi}

\newcommand{\PP}{\mathbf{P}}
\renewcommand{\AA}{\mathbf{A}}

\newcommand{\ZZ}{\mathbf{Z}}

\newcommand{\NN}{\mathbf{N}}
\newcommand{\QQ}{\mathbf{Q}}
\newcommand{\RR}{\mathbf{R}}

\renewcommand{\leq}{\leqslant}

\renewcommand{\geq}{\geqslant}

\newcommand{\ma}{\mathbf}

\newcommand{\x}{\mathbf{x}}
\newcommand{\y}{\mathbf{y}}

\renewcommand{\k}{\mathbf{k}}

\newcommand{\ve}{\varepsilon}

\DeclareMathOperator{\Pic}{Pic}

\DeclareMathOperator{\Mod}{mod} 
\renewcommand{\bmod}[1]{\,(\Mod{#1})}

\DeclareMathOperator{\Val}{Val}

\newcommand{\Q}{{\mathbf{Q}}}
\newcommand{\R}{{\mathbf{R}}}
\newcommand{\Z}{{\mathbf{Z}}}
\newcommand{\N}{{\mathbf{N}}}

\begin{document}
\title{Sums of three squareful numbers}

\author{T.D.  Browning}
\address{School of Mathematics\\
University of Bristol\\ Bristol\\ BS8 1TW\\ United Kingdom}
\email{t.d.browning@bristol.ac.uk}

\author{K. Van Valckenborgh}
\address{Department of Mathematics\\ K.U. Leuven\\ Celestijnenlaan
  200B\\ 3001 Leuven\\ Belgium} 
\email{karl.vanvalckenborgh@wis.kuleuven.be}

\date{\today}

\begin{abstract}
We investigate the frequency  of positive squareful numbers $x,y,z\leq B$ 
for which $x+y=z$ and present a conjecture concerning its asymptotic behaviour.
\end{abstract}
\maketitle

\section{Introduction}

In this paper we examine the quantitative arithmetic of integral
points on certain Campana orbifolds, following the discussions of Abramovich
 \cite{abramovich}, Campana
\cite{Campana} and Poonen \cite{poonen}.
Given rational points $p_i=r_i/s_i \in \PP^1(\QQ)$
with integer multiplicities $m_i\geq 2$, for $1\leq i\leq n$,  we define the divisor 
$
\Delta=\sum_i(1-\frac{1}{m_i})[p_i].
$
The pair $(\PP^1,\Delta)$  defines an orbifold curve in the sense of
Campana and has associated Euler characteristic 
$$
\chi=\chi(\PP^1)-\deg
\Delta=2-n+\frac{1}{m_1}+\cdots + \frac{1}{m_n}.
$$
A point $r/s\in \PP^1(\QQ)$  is said to be integral 
if $rs_i-sr_i$ is $m_i$-powerful for $1\leq i \leq n$.
Here we recall that an integer $k$ is said to be $m$-powerful if
$p^m\mid k$ whenever $p$ is a prime divisor of $k$.
We will focus our attention here upon the orbifold $(\PP^1,\Delta)$ 
associated to the divisor
$$
\Delta=\left(1-\frac{1}{m}\right)[0]+
\left(1-\frac{1}{m}\right)[1]+
\left(1-\frac{1}{m}\right)[\infty],
$$
with Euler characteristic $\chi=-1+\frac{3}{m}.$
The density of integral points on $(\PP^1,\Delta)$ with height at most
$B$ is captured by the counting function
$$
N_{m-1}(B)=\# \left\{
(x,y,z)\in \NN_{\mathrm{prim}}^3: x+y=z, ~x,y,z\leq B, ~
\mbox{$x,y,z$ $m$-powerful}
\right\},
$$
where $\NN$ denotes the set of positive integers and 
$\NN_{\mathrm{prim}}^3$ denotes the set of primitive vectors in $\NN^3$.
%The function $N_{m-1}(B)$ sits between $N_{m-1}'(B)$ and
%$N_{m-1}''(B)$, where the former (respectively, the latter) is defined as for 
%$N_{m-1}(B)$ but with the condition of being $m$-powerful replaced by
%being an $m$th power (respectively, $m$-powered). Here a positive integer $k$ 
%is said to be $m$-powered if $\log k\geq m \sum_{p\mid k} \log p$.
%A discussion of $N_{m-1}''(B)$ can be found in work of Mazur \cite{mazur}.

An old result of Erd\H{o}s and Szekeres \cite{e-g} shows that the
number of $m$-powerful integers up to $x$ is  $c_m
x^{\frac{1}{m}}+O(x^{\frac{1}{m+1}})$, for a certain constant $c_m>0$.
%, with 
%$$
%c_m=\prod_p \left(1+\sum_{h=m+1}^{2m-1}p^{-\frac{h}{m}}\right).
%$$
This leads to a basic trichotomy: we expect 
only finitely many integral points when $\chi<0$, we expect $N_{m-1}(B)$ to
grow at worst logarithmically in $B$ when $\chi=0$ and we expect
$N_{m-1}(B)$ to have order $B^\chi$ when $\chi>0$.
When $m=3$ work of 
Nitaj \cite{nitaj} shows that $N_{2}(B)\gg \log B$.
Our goal in this paper is to provide
evidence in support of the expected order $B^{\frac{1}{2}}$ of
$N_1(B)$ when $m=2$. 

\begin{conjecture}\label{con:1}
We have 
$$
N_1(B)=cB^{\frac{1}{2}}(1+o(1)),
$$ 
as $B\rightarrow \infty$, with $c=2.677539267$ up to eight digits. 
\end{conjecture}

\begin{figure}
  \centering
  \input{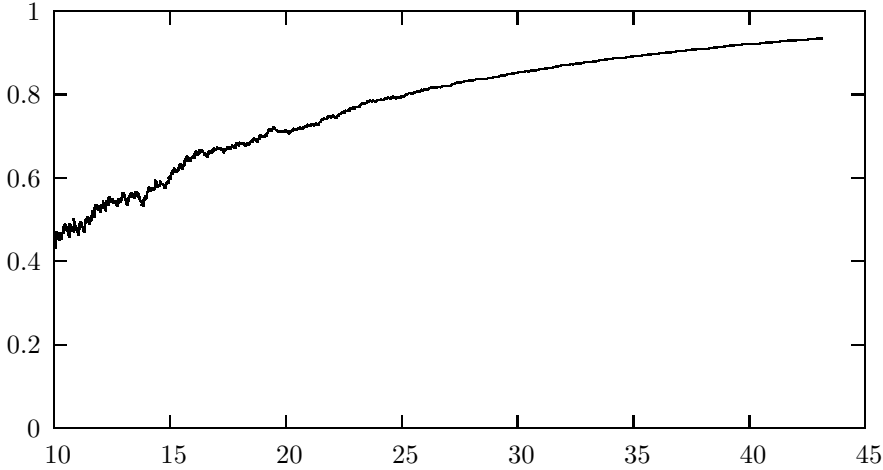}
  \caption{Values of $N_1(B)/(cB^{\frac{1}{2}})$}
  \label{fig:test1}
\end{figure}

The explicit conjectured value of $c$ is too complicated to record
here, but may be found  in \eqref{eq:sonia} and \eqref{eq:constant_program}.
In Figure \ref{fig:test1} the values of $N_1(B)/(cB^{\frac{1}{2}})$ are plotted
for $B$ up to $10^{13}$, where the horizontal 
axis is plotted as $\log_2 B$.
In Table \ref{num} we present some explicit numerical data, including the determination of  the quotient  
$N_1(B)/ (cB^{\frac{1}{2}})$ for large values of $B$.

Any positive squareful integer $k$ can be written uniquely as
$k=x^2y^3$, with $x,y \in \N$ and $y$ square-free. 
Using this description we have
\begin{equation}\label{eq:N1}
N_1(B)=
\sum_{\y \in \NN^3}\mu^2(y_0y_1y_2)
\# \left\{
\x \in \NN^3\cap C_\y: 
\begin{array}{l}
\gcd(x_0y_0,x_1y_1,x_2y_2)=1, \\
x_0^2y_0^3, x_1^2y_1^3, x_2^2y_2^3\leq B
\end{array}
 \right\},
\end{equation}
where $\mu$ is the M\"obius function and 
$C_\y$ denotes the conic
$
x_0^2y_0^3+x_1^2y_1^3=x_2^2y_2^3.
$
One is naturally led to
analyse $N_1(B)$ by counting points on each conic and then summing the
contribution over the $\y$. 
This is the point of view adopted by the second author
\cite{KVV}, where the structure of the  orbifold
$(\mathbf{P}^1,\Delta)$ is generalised to a higher-dimensional
analogue $(\mathbf{P}^{n-1},\Delta)$, corresponding to a hyperplane of
squareful numbers. 
An asymptotic formula of  the expected order of magnitude is then
obtained when there are $n+1\geq 5$ terms present in the hyperplane.
In addition to this \cite{KVV} contains an  interpretation of  the leading constant
 in terms of local densities for the
underlying quadric. We will revisit this discussion in  
\S \ref{s:constant} in order to justify the  numerical value of the constant 
in Conjecture \ref{con:1}.

\begin{table}[h]
\centering
\begin{tabular}{|r|r|r|}
\hline
$B$ & $N_1(B)$ & $
N_1(B)/(cB^{\frac{1}{2}}),
$\\
\hline
$10^7$ & 6562 & 0.774997635 \\
$10^8$ & 21920 & 0.818662130\\
$10^9$ & 72124 & 0.851812396\\
$10^{10}$ & 235168 & 0.878298977\\
$10^{11}$ & 762580 & 0.900636538\\
$10^{12}$ & 2465044 & 0.920637852\\
$10^{13}$ & 7914884 &0.934778480\\
\hline
\end{tabular}
\vspace{0.2cm}
\caption{}\label{num}
\end{table}

Ignoring all but the term with 
 $\y=(1,1,1)$ in \eqref{eq:N1}, one readily  arrives at the lower bound 
$N_1(B)\gg B^{\frac{1}{2}}$, via 
the familiar parametrisation for Pythagorean triples.
Building on 
this observation suitably, we will  sketch a proof of the  following result in \S \ref{s:lower}.

\begin{theorem}\label{t:lower}
We have $N_1(B)\geq c B^{\frac{1}{2}}(1+o(1))$, where $c$ is the constant in Conjecture \ref{con:1}.
\end{theorem}

The problem of producing an upper bound of the expected order of magnitude is much more challenging. In \S \ref{s:upper} we shall establish the following estimate.

\begin{theorem}\label{t:upper}
We have $N_1(B)=O(B^{\frac{3}{5}}\log^{12}B)$.
\end{theorem}

With more work it ought to be possible to remove the factor involving $\log B$ from Theorem~\ref{t:upper}. 
The proof of Theorem \ref{t:upper}  involves two estimates. The first
is based on fixing the $\y$ and counting points on the conic $C_\y$,
uniformly in the coefficients. The second  involves switching
the r\^oles of $\y$ and $\x$, viewing the equation as a family of
plane cubics instead.  
For both of these the determinant method of Heath-Brown \cite{annal}
is a key tool. 
The same argument has been observed by a number of mathematicians, 
including Valentin Blomer in private communication with the first
author. 
In order to improve the exponent of $B$ in Theorem
\ref{t:upper} one requires a new means of treating the contribution
from $\x,\y$ for which each $x_i$ and $y_i$ has order of magnitude
$B^{\frac{1}{5}}$.   
It would be desirable, for example,  to have  better control over the $\y$ which
produce conics  
$C_{\y}$ containing at least one rational point of small height.

\begin{ack}
The authors are grateful to Emmanuel Peyre for useful comments and 
Hendrik Hubrechts for help with preparing the numerical evidence. 
While working on this paper the  first author was 
supported by EPSRC grant number \texttt{EP/E053262/1}.  
\end{ack}

\section{The constant}\label{s:constant}

Recall the expression for $N_1(B)$ in \eqref{eq:N1}, in which 
$C_\y$ denotes the conic 
$$
x_0^2y_0^3+x_1^2y_1^3=x_2^2y_2^3,
$$
for given $\y=(y_0,y_1,y_2)\in \NN^3$.
Let $H_\y: C_\y(\QQ) \rightarrow \R_{\geq 0}$ denote the height function 
$$
[x_0,x_1,x_2]\mapsto \max\{|x_0^2y_0^3|,|x_1^2y_1^3|,|x_2^2y_2^3|\}^{\frac{1}{2}},
$$ 
if $x_0,x_1,x_2 \in \Z$ satisfy 
$\gcd(x_0,x_1,x_2)=1$. 
On noting that $\x$ and $-\x$ represent the same point in $\PP^2$ we easily infer that 
$N_1(B)$ is approximated by the sum
\begin{equation}\label{eq:as_conic}
\frac{1}{4}
\sum_{\y \in \NN^3}\mu^2(y_0y_1y_2)
\#\left\{x\in C_\y(\QQ): H_\y(x)\leq B^{\frac{1}{2}}, ~
\gcd(x_0y_0,x_1y_1,x_2y_2)=1\right\}.
\end{equation}
Following the framework developed by  the second author \cite[\S 5]{KVV},  we 
are therefore led to take the value 
\begin{equation}\label{eq:constant_2}
c=\frac{1}{4}
\sum_{\y\in \NN^3}\mu^2(y_0y_1y_2)  
c_{H_\y}(
C_\y(\mathbb{A}_{\Q})^+),
\end{equation}
in Conjecture \ref{con:1}.
Here, if  $C_\y(\mathbb{A}_{\Q})^+$ denotes the open subset of the adelic space $C_\y(\mathbb{A}_{\Q})$ carved out by the condition
$\min_{0\leq i\leq 2} \{v_{p}(x_{i,p}y_i)\}=0$ for each
prime $p$, then 
$c_{H_\y}(
C_\y(\mathbb{A}_{\Q})^+)$ is 
the constant conjecturally introduced by Peyre \cite[D\'efinition 2.5]{Peyre}.
In particular it follows that 
\begin{equation}\label{eq:constant}
c_{H_\y}(
C_\y(\mathbb{A}_{\Q})^+)
=\alpha(C_\y)
\text{\boldmath{$\omega$}}_{H_\y}(
C_\y(\mathbb{A}_{\Q})^+),
\end{equation}
where 
$\text{\boldmath{$\omega$}}_{H_\y}(
C_\y(\mathbb{A}_{\Q})^+)$
denotes the Tamagawa measure of
$C_\y(\mathbb{A}_{\Q})^+$ associated to the height $H_\y$ and 
$\alpha(C_\y)$ is the volume of a certain polytope contained in the cone of
effective divisors.  

Let $\y\in \NN^3$ with $\mu^2(y_0y_1y_2)=1$. 
In the present setting we have $\Pic(C_\y)\cong \ZZ$ and one finds, using 
\cite[D\'efinition 2.4]{Peyre},  that 
\begin{equation}\label{eq:alpha}
\alpha(C_\y)=\frac{1}{2}.
\end{equation}
In \cite{KVV}, wherein non-singular quadrics in $\PP^{n}$ feature for $n\geq 4$, 
it is worth highlighting that the corresponding value of the constant is found to be $\frac{1}{n-1}$ using the Lefschetz hyperplane 
theorem. This is  no longer true when considering conics in $\PP^2$, since the anticanonical divisor is not a generator for the Picard group.

Turning to the Tamagawa constant 
we let $S=\{\infty,2\}\cup\{p \mid 
y_0y_1y_2\}$, a finite set of places.
The Tamagawa measure on $C_\y(\mathbb{A}_{\Q})$ associated to the height function $H_\y$ is given  by 
\begin{equation}\label{eq:tama}
\text{\boldmath{$\omega$}}_{H_\y}=\lim_{s\to
  1}(s-1)L_S(s,\Pic(\overline{C_\y})) \prod_{v\in
  \Val(\Q)}\lambda_v^{-1}\text{\boldmath{$\omega$}}_{H_{\y},v}, 
\end{equation}
where 
\begin{equation}\label{convergence_factors}
\lambda_v=\begin{cases}
(1-\frac{1}{p})^{-1}, &\mbox{if $v\in \Val(\Q)-S$,}\\ 
1, & \mbox{otherwise}
\end{cases}
\end{equation}
and 
\begin{align*}
L_S(s,\Pic(\overline{C_\y}))
=\prod_{v\in \Val(\Q)-S}\left(1-\frac{1}{p^{s}}\right)^{-1}
=\zeta(s)\prod_{p\mid 2y_0y_1y_2}\left(1-\frac{1}{p^{s}}\right).
\end{align*}
Hence
\begin{align}\label{L-function-explicit}
\lim_{s\to 1}(s-1)L_S(s,\Pic(\overline{C_\y}))
&=\prod_{p\mid 2y_0y_1y_2}\left(1-\frac{1}{p}\right).
\end{align}
In the next few sections, we will calculate the $v$-adic densities at
the different places.

\subsection{Density at the good places}
Let $p$ be a prime such that   $p\nmid 2y_0y_1y_2$. 
Recall that $C_\y(\Q_p)^+$ is defined as the subset of points 
$[x_{0,p},x_{1,p},x_{2,p}]\in C_\y(\Q_p)$, with $x_{i,p}\in \Z_p$ and 
$\min_{0\leq i\leq 2} \{v_p(x_{i,p})\}=0$, for which 
\begin{equation}\label{eq:xy}
\min_{0\leq i\leq 2}
\{v_p(x_{i,p}y_i)\}=0.
\end{equation}
Since  $p\nmid y_0y_1y_2$ this latter condition is automatically
satisfied, whence
$C_\y(\Q_p)^+=C_\y(\Q_p)$.
By Lemmas 3.2 and 3.4 in \cite{peyre3} and  \cite[Lemme 5.4.6]{Peyre}, we have 
$$
\text{\boldmath{$\omega$}}_{H_{\y},p}(C_\y(\Q_p))=\frac{\#C_\y(\mathbf{F}_p)}{p}.
$$ 
Since $C_\y(\mathbf{F}_p)\neq \emptyset$ by Chevalley--Warning, 
we deduce that $\#C_\y(\mathbf{F}_p)=\#\mathbf{P}^1(\mathbf{F}_p)=p+1$. This implies that for the good places we have 
\begin{equation}\label{good_places}
\begin{split}
\prod_{v\in \Val(\Q)-S}\lambda_v^{-1}\text{\boldmath{$\omega$}}_{H_{\y},v}(C_\y (\Q_v)^+)&=\prod_{p\nmid 2y_0y_1y_2}\left(1-\frac{1}{p}\right)\left(1+\frac{1}{p}\right)\\
&=\frac{8}{\pi^2}\cdot \prod_{\substack{p\mid  y_0y_1y_2\\p>2}}\left(
1-\frac{1}{p^2}\right)^{-1},
\end{split}
\end{equation}
since
$\prod_{p>2}\left(1-\frac{1}{p^2}\right)
=\frac{4}{3}\cdot\frac{6}{\pi^2}=\frac{8}{\pi^2}$.

\subsection{Density at the bad places}

We now suppose that   $p$ is a prime divisor of $2y_0y_1y_2$. In this case, when considering $C_\y(\Q_p)^+$, the condition \eqref{eq:xy}
will no longer be satisfied trivially.
Let 
$$
N_{\y}^* (p^r)=\#\left\{\x\in (\Z/p^r\Z)^3-(p\Z/p^r\Z)^3: 
\begin{array}{l}
y_0^3x_0^2+y_1^3x_1^2\equiv y_2^3x_2^2 \bmod{p^r},\\
\min_{0\leq i\leq 2}
\{v_p(x_{i}y_i)\}=0
\end{array}
\right\}.
$$
Using  Lemmas 3.2 and 3.4 in \cite{peyre3}
and  \cite[Lemme 5.4.6]{Peyre}, 
 we deduce the  existence of $r_0\in \NN$ such that 
\begin{equation}\label{property_measure}
\text{\boldmath{$\omega$}}_{H_{\y},p}(C_\y(\Q_p)^+)=\left(1-\frac{1}{p}\right)^{-1}\cdot \frac{N_{\y}^*(p^r)}{p^{2r}}, 
\end{equation}
for each $r\geqslant r_0$.  The following pair of results are concerned with the calculation of 
$N_{\y}^*(p^r)$ for primes $p\mid 2y_0y_1y_2$.

\begin{lemma}\label{lem:bad_place_p}
If $p\mid y_0y_1y_2$ and $p>2$, we have 
$$
\frac{N_{\y}^*(p^r)}{p^{2r}}=
\left(1-\frac{1}{p}\right) \times
\begin{cases}\left(1+\left(\frac{y_1y_2}{p}\right)\right), &
\mbox{ if $p\mid y_0$},\\ 
\left(1+\left(\frac{y_0y_2}{p}\right)\right), &
\mbox{ if $p\mid y_1$},\\
\left(1+\left(\frac{-y_0y_1}{p}\right)\right), &
\mbox{ if $p\mid y_2$}.
\end{cases}
$$
\end{lemma}

\begin{proof}
Suppose, for example, that $p$ divides $y_0$. In this case $p\nmid y_1y_2$.
Modulo $p$ we obtain the equation $y_1^3 x_1^2 \equiv y_2^3x_2^2 \bmod p$. 
If $y_1^{-3}y_2^3$ is a square modulo $p$, then we can choose $x_2$
arbitrarily in $\mathbf{F}_{p}^{\times}$ and for each choice of $x_2$
there are two solutions for $x_1$. It follows that there
are $2p(p-1)$  solutions modulo $p$ in this case. 
If $y_1^{-3}y_2^3$ is not a square modulo $p$, then there are no solutions.
We conclude that $N_{\y}^*(p)=(1+(\frac{y_1y_2}{p}))p(1-p)$. Using Hensel's lemma 
we deduce  that
$N_{\y}^*(p^r)$ is equal to $p^{2(r-1)}(1+(\frac{y_1y_2}{p}))p(1-p)$ for each
$r\geqslant 1$, which thereby completes the proof.  
\end{proof}

\begin{lemma}\label{lem:bad_place_2}
If $r\geqslant 3$, we have
$$
\frac{N_{\y}^*(2^r)}{2^{2r}}=\begin{cases} 
1, &\mbox{ if  $2\nmid y_0y_1y_2$ and $\neg\{y_0 \equiv y_1 \equiv -y_2 \bmod 4\}$},\\ 
2, &\mbox{ if $2\mid y_0$ and $y_1\equiv y_2 \bmod 8$},\\ 
2, &\mbox{ if $2\mid y_1$  and $y_0\equiv y_2 \bmod 8$},\\
2, &\mbox{ if $2\mid y_2$ and $y_0\equiv -y_1 \bmod 8$},\\ 
0, & \mbox{otherwise}. 
\end{cases}
$$
\end{lemma}

\begin{proof}
This follows from direct calculation for the case $r=3$. The formula
for  $r>3$ follows from Hensel's lemma. 
\end{proof}

\subsection{Density at the infinite place}

It remains to consider the infinite place $v=\infty$. Let
\begin{align*}
D_1&=\left\{(y_0^3x_0^2,y_1^3x_1^2,y_2^3x_2^2)\in (\R\cap [-1,1])^3 : y_0^3x_0^2+y_1^3x_1^2=y_2^3x_2^2\right\}.
\end{align*}
Using \cite[Lemme 5.4.7]{Peyre}, we obtain
$$
\text{\boldmath{$\omega$}}_{H_{\y},\infty}
(C_\y(\R)^+)
=\frac{1}{2}\cdot \int_{D_1} \text{\boldmath{$\omega$}}_{L,\infty},
$$
where 
$$
\text{\boldmath{$\omega$}}_{L,\infty} = \frac{\d x_0 \d x_1}{2y_2^{\frac{3}{2}} \sqrt{y_0^3x_0^2+y_1^3x_1^2}}
$$
is the Leray form. 
Let $D_2=\{(x_0,x_1)\in (\R\cap[-1,1])^2: x_0^2+x_1^2\leqslant 1\}$. Then 
it follows that
\begin{equation}\label{infinite_place}
\begin{split}
\text{\boldmath{$\omega$}}_{H_{\y},\infty}
(C_\y(\R)^+)
&= \frac{1}{2}\cdot\frac{1}{(y_0y_1y_2)^{\frac{3}{2}}}\int_{D_2}\frac{1}{\sqrt{x_0^2+x_{1}^2}}\d x_0\d x_1 \\
&=\frac{\pi}{(y_0y_1y_2)^{\frac{3}{2}}}.
\end{split}
\end{equation}

\subsection{Conclusion}

Recall the definition \eqref{eq:tama} of the Tamagawa measure, in
which the convergence factors are given by
\eqref{convergence_factors}. 
Combining \eqref{L-function-explicit}, \eqref{good_places}, \eqref{property_measure}
with 
Lemma \ref{lem:bad_place_p} and \eqref{infinite_place} we deduce that 
$\text{\boldmath{$\omega$}}_{H_{\y}}(C_\y(\mathbb{A}_{\mathbf{Q}})^+)$
is equal to
$$
\frac{1}{(y_0y_1y_2)^{\frac{3}{2}}}\cdot 
\frac{8}{\pi}\cdot
\sigma_{2,\y}\cdot\prod\limits_{\substack{p \mid y_0\\p>2}}\frac{\left(1+\left(\frac{y_1y_2}{p}\right)\right)}{\left(1+\frac{1}{p}\right)}\cdot\prod\limits_{\substack{p \mid y_1\\p>2}}\frac{\left(1+\left(\frac{y_0y_2}{p}\right)\right)}{\left(1+\frac{1}{p}\right)}\cdot \prod\limits_{\substack{p \mid y_2\\p>2}}\frac{\left(1+\left(\frac{-y_0y_1}{p}\right)\right)}{\left(1+\frac{1}{p}\right)}, 
$$
where $\sigma_{2,\y}=\lim_{r\to \infty}
2^{-2r}N_{\y}^*(2^r)$ is given by Lemma \ref{lem:bad_place_2}.
Substituting this into the definition of the conjectural constant
\eqref{eq:constant},  and combining it with \eqref{eq:alpha}, 
we deduce from \eqref{eq:constant_2} that
\begin{equation}\label{eq:sonia}
\begin{split}
c=~&
\frac{1}{\pi}\cdot \sum_{\y\in \NN^3}\frac{\mu^2(y_0y_1y_2)}{(y_0y_1y_2)^{\frac{3}{2}}}
\cdot \sigma_{2,\y}
\\
&\times
\prod\limits_{\substack{p \mid y_0\\p>2}}\frac{\left(1+\left(\frac{y_1y_2}{p}\right)\right)}{\left(1+\frac{1}{p}\right)}\cdot\prod\limits_{\substack{p \mid y_1\\p>2}}\frac{\left(1+\left(\frac{y_0y_2}{p}\right)\right)}{\left(1+\frac{1}{p}\right)}\cdot \prod\limits_{\substack{p \mid y_2\\p>2}}\frac{\left(1+\left(\frac{-y_0y_1}{p}\right)\right)}{\left(1+\frac{1}{p}\right)}. 
\end{split}
\end{equation}

In the remainder of this section we shall attempt to simplify this expression, in order to facilitate a more accurate  numerical computation of it.   Writing 
$
S$ for the set of $\y\in \N^3$ for which $\mu^2(y_0y_1y_2)=1$, 
we can partition $S$ into  subsets
\begin{align*}
S_{-1}&=\{\y\in S : 2\nmid y_0y_1y_2\}, \quad
S_{i}=\{\y\in S : 2\mid y_i\},
\end{align*}
for $0\leq i\leq 2$. We then split \eqref{eq:sonia} into sums $c_i$
over $S_i$, 
for each $-1\leq i\leq 2$.  
To streamline the notation, we define
$$
\gamma(n)=\prod\limits_{p \mid n}\left(1+\frac{1}{p}\right)^{-1}
$$
and, for $a,b \in \N$ with $a,b$ squarefree and $b>1$ odd,  we set
$
(\frac{a}{b})_{*}=1$ if and only if
$(\frac{a}{p})=1$ 
for each $p\mid b$,
with the convention that $(\frac{a}{1})_{*}=1$.

We begin by examining $c_{-1}$, in which case $y_0$, $y_1$ and $y_2$ are all odd. 
We get
\begin{align*}
c_{-1}=~&
\frac{1}{\pi}\cdot\sum_{\substack{\y\in S_{-1}\\ \neg\{y_0 \equiv y_1 \equiv -y_2 \bmod 4\}}}\frac{\gamma(y_0y_1y_2)}{(y_0y_1y_2)^{\frac{3}{2}}}
\\
&\times \prod\limits_{p \mid y_0}\left(1+\left(\frac{y_1y_2}{p}\right)\right)\cdot\prod\limits_{p \mid y_1}\left(1+\left(\frac{y_0y_2}{p}\right)\right)\cdot \prod\limits_{p \mid y_2}\left(1+\left(\frac{-y_0y_1}{p}\right)\right).
\end{align*}
Substituting $d=y_0y_1y_2$, we obtain
$$
c_{-1}=\frac{1}{\pi}\cdot \sum_{\substack{d=1\\2\nmid
    d}}^{\infty}\frac{\mu^2(d)\cdot\gamma(d)\cdot2^{\omega(d)}}{d^{\frac{3}{2}}}\cdot\Delta_{-1}(d), 
$$
where  $\omega(d)$ denotes the number of distinct prime divisors of
$d$ and 
\begin{align*}
\Delta_{-1}(d)=~&
\#\left\{y_0y_1y_2=d :
\begin{array}{l}
\neg\{y_0 \equiv y_1 \equiv -y_2 \bmod 4\},
\\
\left(\frac{y_1y_2}{y_0}\right)_*=\left(\frac{y_0y_2}{y_1}\right)_*=\left(\frac{-y_0y_1}{y_2}\right)_*=1
\end{array}
\right\}.
\end{align*}
We next consider $c_0$, noting that $c_0=c_1=c_2$, by  symmetry. If $y_0$ is even, we set $y_0=2y_0'$, where $y_0'$ is odd. It then holds that 
\begin{align*}
c_{0}=~&
\frac{1}{\pi}\cdot\sum_{\substack{(y_0',y_1,y_2)\in S_{-1}\\ y_1\equiv y_2\bmod 8}}\frac{2\gamma(y_0'y_1y_2)}{(2y_0'y_1y_2)^{\frac{3}{2}}}\\
&\times
\prod\limits_{p\mid y_0'}\left(1+\left(\frac{y_1y_2}{p}\right)\right)\cdot\prod\limits_{p \mid y_1}\left(1+\left(\frac{2y_0'y_2}{p}\right)\right)\cdot
\prod\limits_{p\mid y_2}\left(1+\left(\frac{-2y_0'y_1}{p}\right)\right).
\end{align*}
Putting $d=y_0'y_1y_2$ we deduce as above that  
$$
c_0=\frac{1}{\pi}\cdot\sum_{\substack{d=1\\2\nmid d}}^{\infty}\frac{\mu^2(d)\cdot\gamma(d)\cdot2^{\omega(d)}}{d^{\frac{3}{2}}}\cdot \frac{\Delta_0(d)}{\sqrt{2}},
$$
where now 
$$
\Delta_0(d)=\#\left\{y_0'y_1y_2=d :
\begin{array}{l}
 y_1\equiv y_2 \bmod 8, \\
\left(\frac{y_1y_2}{y_0'}\right)_*=\left(\frac{2y_0'y_2}{y_1}\right)_*=\left(\frac{-2y_0'y_1}{y_2}\right)_*=1
\end{array}
\right\}.
$$
Bringing these expressions together 
in \eqref{eq:sonia}, we conclude that
\begin{equation}\label{eq:constant_program}
c=\frac{1}{\pi}\cdot \sum_{\substack{d=1\\2\nmid d}}^{\infty}\frac{\mu^2(d)\cdot\gamma(d)\cdot2^{\omega(d)}}{d^{\frac{3}{2}}}\cdot\left(\Delta_{-1}(d)+\frac{3}{\sqrt{2}}\Delta_{0}(d)
\right).
\end{equation}
One finds by numerical computation that 
$c=2.677539267$, which is accurate up to eight digits.

\section{The lower bound}\label{s:lower}

Let $C\subset \PP^2$ be a conic defined over $\QQ$ and let $H:C(\QQ)\rightarrow \RR_{\geq 0}$ be an exponential height function. Suppose that $C$ is defined by a non-singular  quadratic form
defined over $\ZZ$ with relatively prime coefficients all bounded in modulus by $M$.
A number of results in the literature are directed at estimating the counting function
$
N_{C,H}(P)=\#\{x\in C(\QQ): H(x)\leq P\},
$
as $P\rightarrow \infty$, with the outcome that there exist absolute constants $\delta, \psi>0$ such that 
\begin{equation}\label{eq:fmt}
N_{C,H}(P)=c_{H}(
C(\mathbb{A}_{\Q})) P+O(M^\psi P^{1-\delta}),
\end{equation}
where 
$c_{H}(
C(\mathbb{A}_{\Q}))$ is the constant predicted by Peyre \cite{Peyre}.
This is a special case of the work of Franke, Manin and Tschinkel \cite{fmt} on flag varieties $P\setminus G$,
with $G$ taken to be  the orthogonal group in three variables.  Typically the uniformity in $M$ is not actually recorded, but it transpires that the dependence on $M$ is at worst polynomial. 

We are now ready to establish Theorem \ref{t:lower}. 
For any choice of $\y$ there are clearly $O(1)$ rational points on $C_\y$  which correspond to a solution with $x_0x_1x_2=0$.  Beginning with \eqref{eq:as_conic} we deduce that 
$$
N_1(B)\geq  
\frac{1}{4}
\sum_{\substack{\y\in \NN^3\\ y_0,y_1,y_2\leq B^\theta}}\mu^2(y_0y_1y_2)   N_{C_\y, H_{\y}}^+ (B^{\frac{1}{2}})  +O(B^{3\theta}),
$$
for any $\theta \leq \frac{1}{3}$, where 
$ N_{C_\y, H_{\y}}^+$ is defined as for $ N_{C_\y, H_{\y}}$, but with the additional constraint that
$\gcd(x_0y_0,x_1y_1,x_2y_2)=1$.  Once taken in conjunction with the fact that $y_0y_1y_2$ is square-free and $\gcd(x_0,x_1,x_2)=1$, we see that the coprimality condition 
$\gcd(x_0y_0,x_1y_1,x_2y_2)=1$ on $C_\y$
 is  equivalent to demanding that 
$\gcd(x_i,x_j,y_k)=1$ for each permutation $\{i,j,k\}=\{0,1,2\}$.
Using the M\"obius function to remove these coprimality conditions gives
$$
N_{C_\y, H_{\y}}^+ (B^{\frac{1}{2}}) =
\sum_{k_0\mid y_0}
\sum_{k_1\mid y_1}
\sum_{k_2\mid y_2} \mu(k_0k_1k_2) 
N_{C_{\k,\y'}, H_{\k,\y'}} \left(\frac{B^{\frac{1}{2}}}{k_0k_1k_2}\right),
$$
where $y_i=k_iy_i'$ for $0\leq i\leq 2$, 
$C_{\k,\y'}$ is the conic 
$k_0y_0'^3x_0^2+k_1y_1'^3x_1^2=k_2y_2'^3x_2^2$ and 
$H_{\k,\y'}$ is defined as for $H_\y$ but with $y_i^3$ replaced by $k_iy_i'^3$, for $0\leq i\leq 2$.
The conic
$C_{\k,\y'}$ has an underlying quadratic form with coefficients of size at most $B^{3\theta}$.
Applying \eqref{eq:fmt} we conclude that 
$$
N_{C_\y, H_{\y}}^+ (B^{\frac{1}{2}}) = B^{\frac{1}{2}}
\sum_{k_0\mid y_0}
\sum_{k_1\mid y_1}
\sum_{k_2\mid y_2} \frac{\mu(k_0k_1k_2) }{k_0k_1k_2}
\cdot
c_{H_{\k,\y'}}(
C_{\k,\y'}(\mathbb{A}_{\Q}))
+O_\ve(B^{\frac{1-\delta}{2}+3\theta\psi+\ve}),
$$
for any $\ve>0$.  One finds that the main term here is precisely equal to 
$c_{H_{\y}}(
C_{\y}(\mathbb{A}_{\Q})^+) B^{\frac{1}{2}}$, in the notation of \S \ref{s:constant}.   
Noting that 
$$
\sum_{y\leq B^{\theta}} \frac{f(y)}{y^{\frac{3}{2}}}=
\sum_{y=1}^\infty \frac{f(y)}{y^{\frac{3}{2}}} +O(B^{-\frac{\theta}{2}+\ve}),
$$
for any arithmetic function $f$ satisfying $f(n)=O_\ve(n^\ve)$, 
we deduce that
$$
N_1(B)\geq  c B^{\frac{1}{2}}
+O(B^{3\theta})
+
O_\ve(B^{\frac{1-\delta}{2}+3\theta(1+\psi)+\ve})
+O_\ve(B^{\frac{1-\theta}{2}+\ve}),
$$
for any $\ve>0$.
We therefore conclude the proof of Theorem \ref{t:lower} by taking  $\theta$ to satisfy
the inequalities
$0<\theta<\frac{\delta}{6(1+\psi)}$.

\section{The upper bound}\label{s:upper}

The aim of this section is to prove Theorem \ref{t:upper}, for which 
our starting point is \eqref{eq:N1}.
In order to estimate $N_1(B)$ we will view the equation in two basic ways: either as a family of conics
or as a family of plane cubic curves.  The work of Heath-Brown \cite{annal} allows one to estimate rational points of bounded height on plane curves, uniformly in the coefficients of the underlying equation. We 
will invoke this theory through the prism of the first author's work \cite[Lemma~4.10]{ferran}, which yields the following bound for any integer $d\geq 2$.

\begin{lemma}\label{lem:ferran}
Let $\ma{c}\in \ZZ^{3}$ with $c_{1}c_{2}c_{3}\neq 0$ and pairwise coprime coordinates. Then we have 
$$
\#\left\{ \mathbf{z}\in \ZZ^{3}: 
\begin{array}{l}
\gcd(z_{1},z_{2},z_{3})=1, ~|z_{i}|\leq Z_{i},\\
c_{1} z_{1}^{d}+c_{2} z_{2}^{d}+c_{3}z_{3}^{d}=0
\end{array}
\right\}
\ll_{d}
\left(1+\frac{Z_{1}Z_{2}Z_{3}}{|c_{1}c_{2}c_{3}|^{\frac{2}{d}}}\right)^{\frac{1}{3}} d^{\omega(c_{1}c_{2}c_{3})}.
$$
\end{lemma}

We will also make use of the familiar bound 
$
\sum_{n\leq x}k^{\omega(n)}\ll x\log^{k-1}x,
$ 
which is valid for any $k\in \NN$. 
We  consider the contribution $N(\ma{X},\ma{Y})$, say, to $N_1(B)$
from $\ma{x}, \ma{y}$ such that  
$$
X_{i}\leq x_{i}<2X_{i}, \quad 
Y_{i}\leq y_{i}<2Y_{i}, 
$$
for $0\leq i\leq 2$.  Clearly $N(\ma{X},\ma{Y})=0$ unless $X_{i}^2Y_{i}^{3}\leq B$ and $X_{i},Y_{i}>1/2$, for $0\leq i\leq 2$. It will be convenient to set 
$X=X_{0}X_{1}X_{2}$ and $Y=Y_{0}Y_{1}Y_{2}$.  In particular we may henceforth assume that 
$
X^{2}Y^{3}\leq B^{3}.
$
On summing over dyadic intervals we see that 
\begin{equation}\label{eq:max}
N_1(B)\ll \log^{6}B \max_{\ma{X},\ma{Y}} N(\ma{X},\ma{Y}),
\end{equation}
where the maximum is over $\ma{X},\ma{Y}$ satisfying the above inequalities. 

Viewing the underlying equation as a family of conics first, we take
$d=2$ in Lemma \ref{lem:ferran} and deduce that 
\begin{align*}
N(\ma{X},\ma{Y})
&\ll \sum_{\y} 2^{\omega(y_{0}y_{1}y_{2})} \left(1+\frac{X}{Y^{3}}\right)^{\frac{1}{3}}\\
&\ll \left(Y+X^{\frac{1}{3}}\right)\log^{3}B. 
\end{align*}
Alternatively, regarding the equation as a family of cubics, we take $d=3$ in Lemma \ref{lem:ferran} and obtain
\begin{align*}
N(\ma{X},\ma{Y})
&\ll \sum_{\x} 3^{\omega(x_{0}x_{1}x_{2})} \left(1+\frac{Y}{X^{\frac{4}{3}}}\right)^{\frac{1}{3}}\\
&\ll \left(X+Y^{\frac{1}{3}}X^{\frac{5}{9}}\right)\log^{6}B. 
\end{align*}
Bringing these two estimates together we conclude that 
$$
N(\ma{X},\ma{Y})
\ll \left(\min \{X,Y\} + \min\{Y, Y^{\frac{1}{3}}X^{\frac{5}{9}}\}+X^{\frac{1}{3}} \right)\log^{6}B. 
$$
Now it is clear that $\min\{X,Y\}\leq X^{\frac{2}{5}}Y^{\frac{3}{5}}\leq B^{\frac{3}{5}}$ and 
$$
\min\{Y, Y^{\frac{1}{3}}X^{\frac{5}{9}}\}
\leq Y^{\frac{9}{25}}\cdot
(Y^{\frac{1}{3}}X^{\frac{5}{9}})^{\frac{18}{25}} =
X^{\frac{2}{5}}Y^{\frac{3}{5}}
\leq B^{\frac{3}{5}},
$$
since $X^2Y^3\leq B^{3}$. 
Finally  we note that 
$X^{\frac{1}{3}}\leq B^{\frac{1}{2}}$. Inserting our estimate for $N(\ma{X},\ma{Y})$ into \eqref{eq:max},  we therefore arrive at the statement of Theorem \ref{t:upper}.

\end{document}